\documentclass[12pt,oneside,reqno]{amsart}

  \usepackage{a4wide}
  \usepackage{amssymb}
  \usepackage{amsmath}
  \usepackage{amsfonts}
  \usepackage{amscd}
  \usepackage{xspace}
  \usepackage{boxedminipage}
  \usepackage{nomencl}

  \usepackage{booktabs}
  \usepackage{multirow}
  \usepackage[figuresright]{rotating}
  \usepackage[format=hang,font=small]{caption} 
  \usepackage[curve,graph]{xypic}
  \usepackage{lscape}
  \usepackage{subfigure}
\usepackage[T1]{fontenc}
\usepackage[utf8]{inputenc}

  \usepackage{stmaryrd}
\def\NN{{\mathbb{N}}}

\def\RR{{\mathbb{R}}}
\def\CC{{\mathbb{C}}}

\def\deg{{\mathrm{deg}}}
\def\rank{{\mathrm{rank}}}
\def\ch{{\mathrm{ch}}}
\def\Ind{{\mathrm{Ind}}}
\def\min{{\mathrm{min}}}
\def\codim{{\mathrm{codim}}}

  \usepackage{tikz}
  \usetikzlibrary{calc}

  \makeatletter
  \CheckCommand*\refstepcounter[1]{\stepcounter{#1}%
      \protected@edef\@currentlabel
       {\csname p@#1\endcsname\csname the#1\endcsname}%
  }
  \renewcommand*\refstepcounter[1]{\stepcounter{#1}%
    \protected@edef\@currentlabel
      {\csname p@#1\expandafter\endcsname\csname the#1\endcsname}%
  }
  \def\labelformat#1{\expandafter\def\csname p@#1\endcsname##1}
  \DeclareRobustCommand\Ref[1]{\protected@edef\@tempa{\ref{#1}}%
     \expandafter\MakeUppercase\@tempa
  }
  \makeatother

  \makeatletter
  \newcommand{\numberlike}[2]{%
     \expandafter\def\csname c@#1\endcsname{%
         \expandafter\csname c@#2\endcsname}%
  }
  \makeatother

  \def\DefaultNumberTheoremWithin{section}

  \theoremstyle{plain}
  \newtheorem{Lemma}{Lemma}
     \numberwithin{Lemma}{\DefaultNumberTheoremWithin}
     \labelformat{Lemma}{Lemma~#1}
  
     \numberwithin{Claim}{\DefaultNumberTheoremWithin}
     \numberlike{Claim}{Lemma}
     \labelformat{Claim}{Claim~#1}
  \newtheorem{Theorem}{Theorem}
     \numberwithin{Theorem}{\DefaultNumberTheoremWithin}
     \numberlike{Theorem}{Lemma}
     \labelformat{Theorem}{Theorem~#1}
  
     \numberwithin{Corollary}{\DefaultNumberTheoremWithin}
     \numberlike{Corollary}{Lemma}
     \labelformat{Corollary}{Corollary~#1}
  
     \numberwithin{Proposition}{\DefaultNumberTheoremWithin}
     \numberlike{Proposition}{Lemma}
     \labelformat{Proposition}{Proposition~#1}
  
     \numberwithin{Conjecture}{\DefaultNumberTheoremWithin}
     \numberlike{Conjecture}{Lemma}
     \labelformat{Conjecture}{Conjecture~#1}

  \theoremstyle{definition}
  
     \numberwithin{Definition}{\DefaultNumberTheoremWithin}
     \numberlike{Definition}{Lemma}
     \labelformat{Definition}{Definition~#1}

  \theoremstyle{definition}
  \newtheorem{Question}{Question}
     \numberwithin{Question}{\DefaultNumberTheoremWithin}
     \numberlike{Question}{Lemma}
     \labelformat{Question}{Question~#1}

  \theoremstyle{definition}
  
     \numberwithin{Problem}{\DefaultNumberTheoremWithin}
     \numberlike{Problem}{Lemma}
     \labelformat{Problem}{Problem~#1}

  \theoremstyle{remark}
  
     \numberwithin{Remark}{\DefaultNumberTheoremWithin}
     \numberlike{Remark}{Lemma}
     \labelformat{Remark}{Remark~#1}
  \theoremstyle{remark}

     \numberwithin{Example}{\DefaultNumberTheoremWithin}
     \numberlike{Example}{Lemma}
     \labelformat{Example}{Example~#1}
  
     \labelformat{Case}{Case~#1}
     \numberwithin{Case}{Lemma}
  
     \labelformat{Step}{Step~#1}
     \numberwithin{Step}{Lemma}
     
     \theoremstyle{table}
  \newtheorem{Table}{Table}
     \numberwithin{Table}{\DefaultNumberTheoremWithin}
     \numberlike{Table}{Lemma}
     \labelformat{Table}{Table~#1}
  \theoremstyle{table}

  \labelformat{equation}{(#1)}
  \labelformat{figure}{Figure~#1}
  \labelformat{section}{\S#1}
  \labelformat{subsection}{\S#1}
  \labelformat{subsubsection}{\S#1}

  \def\eqref{\ref}

  \allowdisplaybreaks

\begin{document}


 \title[Representation stability on the cohomology of complements]{Representation stability on the cohomology of complements of subspace arrangements}


  \author{Artur Rapp}
     \address{Fachbereich Mathematik und Informatik\\
              Philipps-Universit\"at Marburg\\
              35032 Marburg\\
              Germany}
     \email{rapp202@mathematik.uni-marburg.de}

  \thanks{I thank Nir Gadish for helpful comments.}

  \begin{abstract}
  We study representation stability in the sense of Church and Farb of 
sequences of cohomology groups of
complements of arrangements of linear subspaces in real and complex 
space as $S_n$-modules. We consider
arrangement of linear subspaces defined by sets of diagonal equalities 
$x_i = x_j$ and invariant under the
action of $S_n$ permuting the coordinates. We provide bounds on the point 
when stabilization occurs and
an alternative proof for the fact that stabilization happens. The latter 
is a special case of a very general
stabilization result of Gadish and for the pure braid space the result 
is part of the work of Church and Farb.
For this space better stabilization bounds were obtained by Hersh and Reiner.

  \end{abstract}

  \maketitle

\section{Introduction}
           \label{sec:introduction}   
In this paper, we consider arrangements of diagonal subspaces of $\RR^{dn}$ for natural numbers $d$ and $n$. Let $\pi$ be a set partition of $\{1,...,n\}$. Let $W_{\pi}^d$ be the linear subspace of $n$-tuples $(w_1,...,w_n)$ of points in $\RR^d$ such that $w_i=w_j$ whenever $i$ and $j$ are in the same block of $\pi$. For an integer partition $\lambda$ we denote by $\mathcal{A}_{\lambda}^d$ the arrangement of all subspaces $W_{\pi}^d$ such that $\pi$ is of type $\lambda$. More generally, set $\mathcal{A}_{\Lambda}^d=\cup_{\lambda\in\Lambda} \mathcal{A}_{\lambda}^d$ for every finite set $\Lambda$ of integer partitions of $n$. The complement $\mathcal{M}_{\Lambda}^d=\RR^{dn}\setminus \cup_{W\in\mathcal{A}_{\Lambda}^d}W$ is a real manifold. If $\Lambda=\{\lambda\}$, we write $\mathcal{M}^d_{\lambda}$ for $\mathcal{M}^d_{\Lambda}$. All representations, homology groups and cohomology groups in this paper are taken with coefficients in $\CC$. The action of the symmetric group $S_n$ on $n$-tuples of points in $\RR^d$ by permuting the coordinates induces an $S_n$-representation on the reduced singular cohomology $\tilde H^i(\mathcal{M}_{\Lambda}^d)$. We look into representation stability in the sense of Church and Farb (see \cite{CF}) of these modules. Our main purpose is to prove that sequences of these modules stabilize and to obtain stabilization bounds. This is the content of \ref{main}. The fact that these sequences stabilize is a special case of a result of Gadish (\cite[Theorem A]{G}). The case $\Lambda=\{(2,1^{n-2})\}$ was proved by Church (\cite[Theorem 1]{C}) and for this case Hersh and Reiner provided better stabilization bounds (\cite[Theorem 1.1]{HR}). For an integer partition $\lambda$ we write $l(\lambda)$ for its length. As in \cite[Definition 2.5]{HR} let $\rank(\lambda):=|\lambda|-l(\lambda)$ be the rank of $\lambda$.
\newpage
\begin{Theorem}\label{main}
Let $\Lambda$ be a nonempty finite set of integer partitions of the number $n_0$ not containing $(1^{n_0})$. For every $n\ge n_0$ let $\Lambda^{(n)}$ be the set of all integer partitions of $n$ obtained from an integer partition in $\Lambda$ by adding $n-n_0$ parts of size $1$. Let $\rank(\Lambda)=\min\{\rank(\lambda)~|~\lambda\in \Lambda\}$. For every $i$ and $d\ge 2$ the sequence $\{\tilde H^i(\mathcal{M}^d_{\Lambda^{(n)}})\}_n$ stabilizes at $4(i+1-\rank(\Lambda))/(d-1)$. 
\end{Theorem}
In Section 2, we provide the definition of representation stability and prove \ref{main}. In Section 3, we consider the special case $\mathcal{M}^d_{(k,1^{n-k})}$ for $k\ge d+1$. We prove that stability in this case starts earlier than in the bound given in \ref{main}.  


\section{Proof of the main theorem}
           \label{Proof of the main theorem}
An integer partition $\lambda$ of a natural number $n$ is a finite sequence $(\lambda_1,\lambda_2,...)$ with $\lambda_1\ge \lambda_2\ge...$ and $\sum_{i\ge 1}\lambda_i=n$. We sometimes denote $\lambda$ by $(1^{m_1(\lambda)},...,n^{m_n(\lambda)})$ where $m_i(\lambda)$ is the number of occurences of the number $i$ in $\lambda$ for $1\le i\le n$.
Given an integer partition $\lambda=(\lambda_1,\lambda_2,...)$ of $n$ we write $\lambda+\square$ to denote $(\lambda_1+1,\lambda_2,\lambda_3,...)$.
Let $V$ be an $S_n$-representation and $\sum_{\lambda\vdash n} a_{\lambda}S^{\lambda}$, $a_{\lambda}\in\CC$, its decomposition into irreducible $S_n$-representations $S^{\lambda}$. Then we write
$$ V+\square:= \sum_{\lambda\vdash n} a_{\lambda}S^{\lambda+\square}. $$
We use the same notation, if we replace $S_n$-representations by symmetric functions (see \cite{M} for background on symmetric functions). As in \cite{M} we write $\ch$ for the Frobenius characteristic and $s_{\lambda}$ for the Schur function indexed by the integer partition $\lambda$. The above equation becomes
$$ \ch(V)+\square:= \sum_{\lambda\vdash n} a_{\lambda}s_{\lambda+\square}. $$
Now let $n_0\in\NN$. Let $\{V_n\}_{n\ge n_0}$ be a sequence of $S_{n+n_0}$-representations or a sequence of characteristics of $S_{n+n_0}$-representations. We say that this sequence stabilizes at $m\ge n_0$, if
$$ V_n=V_{n-1}+\square $$
for all $n>m$. We say that the sequence stabilizes sharply at $m$, if $m$ is the smallest integer such that
$$ V_n=V_{n-1}+\square $$ for all $n>m$.
The following lemma is a generalization of \cite[Lemma 2.2]{HR}. 

\begin{Lemma}\label{productwithsnalpha}
Let $\lambda$ and $\alpha$ be integer partitions. For every $n\ge \alpha_1$ we consider the integer partition $(n,\alpha)=(n,\alpha_1,\alpha_2,...)$. The sequence $\{s_{(n,\alpha)}s_{\lambda}\}_{n}$ stabilizes sharply at $\lambda_1+\alpha_1$. In other words
$$ s_{(n,\alpha)}s_{\lambda}=s_{(n-1,\alpha)}s_{\lambda}+\square $$ if and only if $n>\lambda_1+\alpha_1$.
\end{Lemma}
\begin{proof}
Suppose $n>\lambda_1+\alpha_1$. Let $\nu$ be an integer partition of $n+|\lambda|+|\alpha|$ with $(n,\alpha)\subseteq \nu$. Let $LR_{n,\lambda}^{\nu}$ be the set of all Littlewood-Richardson tableaux of shape
$\nu/(n,\alpha)$ and weight $\lambda$. By the Littlewood-Richardson rule (see \cite{M}) the multiplicity of $s_{\nu}$ in $s_{(n,\alpha)}s_{\lambda}$ is $\#LR_{n,\lambda}^{\nu}$. Let $\nu'$ be the integer partition of $n-1+|\lambda|+|\alpha|$ obtained from $\nu$ by replacing $\nu_1$ by $\nu_1-1$. We define the map
$$ \phi:LR_{n,\lambda}^{\nu}\rightarrow LR_{n-1,\lambda}^{\nu'} $$
by the following procedure: Remove the first empty box in the first row of the tableau and then move all other boxes of the first row one place to the left. The two steps are illustrated below with $n=5,~\alpha=(1,1),~\lambda=(3,1)$ and $\nu=(6,4,1)$:
$$
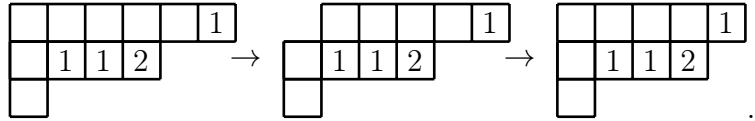
\begin{tikzpicture}[scale=0.5, line width=1pt]
  \draw (0,0) grid (6,1);
  \draw (0,0) grid (4,-1);
  \draw (0,0) grid (1,-2);
  \node[left] (A) at (6,0.5) {1};
  \node[left] (A) at (2,-0.5) {1};
  \node[left] (A) at (3,-0.5) {1};
  \node[left] (A) at (4,-0.5) {2};
  \node[left] (A) at (7,-0.5) {$\rightarrow$};
  \end{tikzpicture}
~
\begin{tikzpicture}[scale=0.5, line width=1pt]
  \draw (1,0) grid (6,1);
  \draw (0,0) grid (4,-1);
  \draw (0,0) grid (1,-2);
  \node[left] (A) at (6,0.5) {1};
  \node[left] (A) at (2,-0.5) {1};
  \node[left] (A) at (3,-0.5) {1};
  \node[left] (A) at (4,-0.5) {2};
   \node[left] (A) at (7,-0.5) {$\rightarrow$};
\end{tikzpicture}
~
\begin{tikzpicture}[scale=0.5, line width=1pt]
  \draw (0,0) grid (5,1);
  \draw (0,0) grid (4,-1);
  \draw (0,0) grid (1,-2);
  \node[left] (A) at (5,0.5) {1};
  \node[left] (A) at (2,-0.5) {1};
  \node[left] (A) at (3,-0.5) {1};
  \node[left] (A) at (4,-0.5) {2};
\end{tikzpicture}.
$$
We want to show that the resulting tableau is indeed a Littlewoood-Richardson tableau so that $\phi$ is well defined. The only condition that has to be checked is that we have no two $1$'s in the first two rows that lie above each other. But this follows from $n>\lambda_1+\alpha_1$, because this implies that $n$ is larger than the number of empty boxes in the second row plus the number of $1$'s in the second row. $\phi$ has an inverse map: Given a tableau in $LR_{n-1,\lambda}^{\nu'}$ we move the first row one place to the right and put an empty box in the gap. So $\phi$ is bijective and $\#LR_{n,\lambda}^{\nu}=\#LR_{n-1,\lambda}^{\nu'}$. This shows that $\{s_{(n,\alpha)}s_{\lambda}\}_{n}$ stabilizes at $\lambda_1+\alpha_1$.
Now let $n=\lambda_1+\alpha_1$ and $\nu=(n,n,\lambda_2+\alpha_2,\lambda_3+\alpha_3,...)$ There is a Littlewood-Richardson tableau of shape $\nu/(n,\alpha)$ and weight $\lambda$: We look at the Ferrer diagram of $\nu$ and put $\lambda_1$ many $1$'s at the end of the second row, $\lambda_2$ many $2$'s at the end of the third row and so on. It follows that we have a Schur function $s_{\nu}$ with $\nu_1=\nu_2$ with multiplicity $\ge 1$ in the decomposition of $s_{(n,\alpha)}s_{\lambda}$. We have
$$ s_{(n,\alpha)}s_{\lambda}\neq f+\square $$
for every symmetric function $f$.
\end{proof}
For a finite arrangement $\mathcal{A}$ of linear subspaces of $\RR^{dn}$ the intersection lattice $L_{\mathcal{A}}$ is the set of intersections of arbitrarily many elements of $\mathcal{A}$ ordered by reversed inclusion. The least element $\hat 0$ is $\RR^{dn}$ and the greatest element $\hat 1$ is the intersection of all elements of $\mathcal{A}$. For a subset $T$ of $L_{\mathcal{A}}$ the join sublattice of $L_{\mathcal{A}}$ generated by $T$ consists of all nonempty intersections of arbitrarily many elements of $T$ also ordered by reversed inclusion. If $\mathcal{A}$ is an arrangement of diagonal subspaces given by equations of the form $w_i=w_j$ for $1\le i<j\le n$, $w=(w_1,...,w_n)\in (\RR^d)^n$, the intersection lattice $L_{\mathcal{A}}$ is isomorphic to the lattice $\Pi_n$ of set partitions of $\{1,...,n\}$ ordered by refinement. For a set partition $\pi\in\Pi_n$ we also write $\pi$ for the corresponding subspace of $\RR^{dn}$. If $\pi\in \Pi_n$ is a set partition into the subsets $B_1,...,B_l$ of $\{1,...,n\}$, we write $\pi=B_1|...|B_l$. In this notation, we have $\hat 0=\{1\}|\{2\}|...|\{n\}$. $\pi$ is said to be finer than $\pi'=C_1|...|C_m$, if for every $1\le i\le l$ there is a $1\le j\le m$ such that $B_i\subseteq C_j$. We may renumber the sets $B_1,...,B_l$ such that $\#B_1\ge...\ge \#B_l$. The integer partition $(\#B_1,...,\#B_l)$ is then called the type of $\pi$. If $\Lambda$ is a set of integer partitions of $n$, then $\Pi_{\Lambda}$ is the join sublattice of $\Pi_n$ generated by all set partitions of type $\lambda$ for all $\lambda\in\Lambda$. 
\begin{proof}[Proof of \ref{main}]
By \cite[Theorem 2.5(ii)]{SW} we have
$$ \tilde H^i(\mathcal{M}_{\Lambda^{(n)}}^d)=\bigoplus_{\pi\in (\Pi_{\Lambda^{(n)}}^{>\hat 0})/S_n} \Ind_{(S_n)_{\pi}}^{S_n}(\tilde H_{\codim(\pi)-i-2}([\hat 0,\pi])\otimes \tilde H_{\codim(\pi)-1}(S^{dn-1}\cap \pi^{\perp})). $$
$(\Pi_{\Lambda^{(n)}}^{>\hat 0})/S_n$ is a set of representatives of the action of $S_n$ on $\Pi_{\Lambda^{(n)}}$ excluding $\hat 0$, $(S_n)_{\pi}$ is the stabilizer subgroup of $\pi$, $\tilde H_j([\hat 0,\pi])$ is the reduced simplicial homology on the order complex $\Delta([\hat 0,\pi])$ in degree $j\ge -1$, $\codim(\pi)$ is the codimension of $\pi$ as a real subspace of $\RR^{dn}$ and $S^{dn-1}$ is the $(dn-1)$-dimensional sphere.
If $\pi$ is of type $\mu=(1^{m_1(\mu)},2^{m_2(\mu)},...)\vdash n$, then its stabilizer $(S_n)_{\pi}$ is the product of wreath products $\times_j S_{m_j(\mu)}[S_j]$ and $\codim(\pi)=d(n-l(\mu))=d\cdot \rank(\mu)$. The length of a chain in $\Pi_n$ from $\hat 0$ to $\pi$ is $\le \sum_{j=1}^{l(\mu)} (\mu_j-1) = n-l(\mu)=\rank(\mu)$. Since the atoms in $\Pi_{\Lambda}$ are of shape $\lambda$ for $\lambda\in \Lambda$, the length of a chain in $\Pi_{\Lambda}$ from $\hat 0$ to $\pi$ has length $\le \rank(\mu)-\rank(\Lambda)+1$ and contributes to homology in degree $\le \rank(\mu)-\rank(\Lambda)-1$.
It follows that if the homology $\tilde H_{\codim(\pi)-i-2}([\hat 0,\pi])$ is not zero, then
$$ -1\le d\cdot \rank(\mu)-i-2\le \rank(\mu)-\rank(\Lambda)-1  $$
and then
$$ (i+1)/d\le \rank(\mu)\le (i+1-\rank(\Lambda))/(d-1). $$
Let $\tilde \mu$ be the integer partition obtained from $\mu$ by removing the parts of size $1$. The rank of $\mu$ and the rank of $\tilde \mu$ are the same.
From \cite[Proposition 2.8]{HR}, we have $\rank(\tilde\mu)+1\le |\tilde\mu|\le 2\cdot\rank(\tilde\mu)$. This yields
$$ 1+(i+1)/d\le |\tilde \mu|  \le 2(i+1-\rank(\Lambda))/(d-1). $$
The subgroup $S_{m_1(\mu)}[S_1]\cong S_{m_1(\mu)}$ acts trivially on $\tilde H_{\codim(\pi)-i-2}([\hat 0,\pi])$. The coordinates of vectors in the space $\pi^{\perp}$ which correspond to the singletons of $\pi$ are zero. It follows that the above copy of $S_{m_1(\mu)}$ acts trivially on $H_{\codim(\pi)-1}(S^{dn-1}\cap \pi^{\perp})$. Let $S^{(m_1(\mu))}$ be the trivial $S_{m_1(\mu)}$-module. We get the following isomorphism of $\times_{j\ge 1} S_{m_j(\mu)}[S_j]$-modules:
 $$\tilde H_{\codim(\pi)-i-2}([\hat 0,\pi])\otimes \tilde H_{\codim(\pi)-1}(S^{dn-1}\cap \pi^{\perp}) \cong $$ $$ S^{(m_1(\mu))}\otimes (\tilde H_{\codim(\pi)-i-2}([\hat 0,\pi])\otimes \tilde H_{\codim(\pi)-1}(S^{dn-1}\cap \pi^{\perp})). $$
We consider the interval $[\hat 0,\pi]$ in $\Pi_{\Lambda^{(n)}}$. The atoms in $[\hat 0,\pi]$ have $\ge n-n_0$ singletons. If we delete $\min\{n-|\tilde\mu|,n-n_0\}$ many singletons from $\pi$, after renumbering we can view $[0,\pi]$ as an interval in $\Pi_{\Lambda^{(max\{|\tilde\mu|,n_0\})}}$. We can also forget about the coordinates of vectors in $\pi^{\perp}$ which correspond to the singletons of $\pi$. We have $\codim(\pi)=d\cdot \rank(\tilde\mu)$. It follows that the $\times_{j\ge 2}S_{m_j(\mu)}[S_j]$-module
$$ \tilde H_{\codim(\pi)-i-2}([\hat 0,\pi])\otimes \tilde H_{\codim(\pi)-1}(S^{dn-1}\cap \pi^{\perp}) $$  
does not depend on $n$ and we write $V_{\tilde \mu}$ for it.
Using the transitivity of induction on $\times_{j\ge 1} S_{m_j(\mu)}[S_j]\le S_{m_1(\mu)}\times S_{n-m_1(\mu)}\le S_n$ we get:
$$ \Ind_{\times_{j\ge 1}S_{m_j(\mu)}[S_j]}^{S_n}(S^{m_1(\mu)}\otimes V_{\tilde \mu})= $$
$$ \Ind_{S_{m_1(\mu)}\times S_{n-m_1(\mu)}}^{S_n}(S^{m_1(\mu)}\otimes Ind_{\times_{j\ge 2}S_{m_j(\mu)}[S_j]}^{S_{n-m_1(\mu)}}(V_{\tilde \mu}))= $$
$$ \Ind_{S_{n-|\tilde \mu|}\times S_{|\tilde\mu|}}^{S_n}(S^{n-|\tilde\mu|}\otimes \Ind_{\times_{j\ge 2}S_{m_j(\tilde\mu)}[S_j]}^{S_{|\tilde\mu|}}(V_{\tilde \mu})). $$
Let
$$ f_{\tilde\mu}:= \ch(\Ind_{\times_{j\ge 2}S_{m_j(\tilde\mu)}[S_j]}^{S_{|\tilde\mu|}}(V_{\tilde \mu})).$$
We have 
$$ \ch(\Ind_{S_{n-|\tilde \mu|}\times S_{|\tilde\mu|}}^{S_n}(S^{n-|\tilde\mu|}\otimes \Ind_{\times_{j\ge 2}S_{m_j(\tilde\mu)}[S_j]}^{S_{|\tilde\mu|}}(V_{\tilde \mu})))= $$
$$ h_{n-|\tilde\mu|}f_{\tilde\mu} $$
where $h_{n-|\tilde\mu|}=s_{(n-|\tilde\mu|)}$.
It follows that the characteristic of $\tilde H^i(\mathcal{M}_{\Lambda^{(n)}}^d)$ is
$$
\sum_{{{\tilde\mu~an~integer~partition~with~no~parts~of~size~1,}\atop{1+(i+1)/d\le |\tilde \mu|\le 2(i+1-\rank(\Lambda))/(d-1)}}} h_{n-|\tilde\mu|}f_{\tilde\mu}. $$
From \ref{productwithsnalpha}, it follows that the sequence stabilizes at a number being larger than $2|\tilde\mu|$ for every $\tilde\mu$ occuring in the sum. This is fulfilled at $4(i+1-\rank(\Lambda))/(d-1)$.
\end{proof}

\section{Improved stability bounds for $k$-equal arrangements}
           \label{Improved stability bounds for $k$-equal arrangements}
	   
	   We consider the sequence $\{\tilde H^i(\mathcal{M}_{(2,1^{n-2})}^d)\}_n$. \ref{main} states that stabilization occurs at $4i/(d-1)$. Compare this to the known results: By \cite[Theorem 1]{C} we have stabilization at $2i$ for $d\ge 3$ and stabilization at $4i$ for $d=2$. By \cite[Theorem 1.1]{HR} we have the following for $i\ge 1$. The sequence is zero from the beginning, if $d-1$ does not divide $i$. Otherwise it stabilizes sharply at $3i/(d-1)$ for odd $d\ge 3$ and it stabilizes sharply at $3i/(d-1)+1$ for even $d\ge 2$. Now we consider the sequence $\{\tilde H^i(\mathcal{M}_{(k,1^{n-k})}^d)\}_n$ for general $k\ge 2$.
	   The stability of this sequence was also considered by Gadish  (\cite[Example 6.11]{G}) as an example of his general results. 
	   We want to determine lower bounds than the ones given in \ref{main} where stabilization occurs for $k\ge d+1$. Let $h_n=s_{(n)}$ be the complete homogeneous symmetric function, $e_n=s_{(1^n)}$ the elementary symmetric function and $\omega$ the involutive ring homomorphism of the ring of symmetric functions with $\omega(h_n)=e_n$. We write $\pi_n$ for the characteristic of $\tilde H^{n-3}(\Delta(\Pi_n))$ and $l_n=\omega(\pi_n)$. For symmetric functions $f$ and $g$ we write $f[g]$ for the plethysm of these two functions.
\begin{Lemma}\label{psi}
Let $d\ge 2,~k\ge d+1,i\ge 0$ and $n\ge 1$. Let $r,t\ge 1$ and $q\ge 0$ be such that $i=(d-1)(n-r-q)+t(k-2)$. Let $U_k:=\sum_{j\ge k}s_{(j-k+1,1^{k-1})}$ and $\psi_{n,q,r,t}$ be 
$$ \begin{cases}
     \omega\left(\omega^k\left(e_r[\sum_{j\ge 1}l_j]\right)|_{\deg~t}[U_k]\right)|_{\deg ~n-q}h_q & if~d~is~even \\ \\
     \left(\left(h_r[\sum_{j\ge 1}l_j]\right)|_{\deg~t}[U_k]\right)|_{\deg~n-q}h_q & if~d~is~odd~and~k~is~even  \\ \\
     \left(\left((-1)^th_r[\sum_{j\ge 1}(-1)^j\pi_j]\right)|_{\deg~t}[U_k]\right)|_{\deg~n-q}h_q & if~d~and~k~are~odd 
   \end{cases}. $$
   Then
\begin{itemize}
\item[(i)] $\psi_{n,q,r,t}=\psi_{n-1,q-1,r,t}+\square$ if $q>n/2$ and $n\ge 2$.
\item[(ii)] $\psi_{n,q,r,t}=\psi_{n-1,q-1,r,t}+\square$ if $d$ is even, $q>tk$ and $n\ge 2$.
\item[(iii)] $\psi_{n,q,r,t}=0$ if $r>t$ or $t>n/k$.
\item[(iv)] $\psi_{n,q,r,t}=0$ if $q\le n/2$ and $n>\frac{2i}{d-1}$.
\item[(v)] $\psi_{n,q,r,t}=0$ if $k\ge d+2$, $q\le tk$ and $n>\frac{ki}{k-d-1}$.
\end{itemize}   
\end{Lemma}
\begin{proof}
(i) We have $\psi_{n,q,r,t}=f_{n-q}h_q$ for a symmetric function $f_{n-q}$ of degree $n-q$ and $h_q=s_{(q)}$. From \ref{productwithsnalpha}, we get 
$$ \psi_{n,q,r,t}=f_{n-q}h_q=f_{n-q}h_{q-1}+\square=f_{(n-1)-(q-1)}h_{q-1}+\square=\psi_{n-1,q-1,r,t}+\square $$ 
if $q>n-q$ or equivalently $q>n/2$.\\
(ii) If $d$ is even, then $\psi_{n,q,r,t}=\omega(f_t[U_k])|_{deg~n-q}h_q$ for a symmetric function $f_t$ of degree $t$. The first column of every Schur function in $U_k=\sum_{j\ge k}s_{(j-k+1,1^{k-1})}$ has length $k$. From \cite[Proposition 4.3 (d)]{HR} follows that for every $s_{\lambda}$ with $\lambda\vdash n-q$ occuring in the Schur function decomposition of $\omega(f_t[U_k])|_{deg~n-q}$ the first row of $\lambda$ has length $\le tk$. If $q>tk$, it follows from \ref{productwithsnalpha} that
$$ \psi_{n,q,r,t}=\omega(f_t[U_k])|_{deg~n-q}h_q=
\omega(f_t[U_k])|_{deg~n-q}h_{q-1}+\square= $$  $$
\omega(f_t[U_k])|_{deg~(n-1)-(q-1)}h_{q-1}+\square=\psi_{n-1,q-1,r,t}+\square. $$
\\
(iii) If $r>t$ the terms $e_r[\sum_{j\ge 1}l_j],~h_r[\sum_{j\ge 1}l_j]$ and
$(-1)^th_r[\sum_{j\ge 1}(-1)^j\pi_j]$ only have terms of degree $>t$. Then the whole term $\psi_{n,q,r,t}$ is zero. $U_k$ only has terms of degree $\ge k$. Then $f_t[U_k]$ for a symmetric function $f_t$ of degree $t$ has only terms of degree $\ge tk$. If $t>n/k$ then $tk>n\ge n-q$ and again $\psi_{n,q,r,t}$ is zero. \\
(iv) Suppose $\psi_{n,q,r,t}\neq 0$. We have to show that $q>n/2$ or $n\le \frac{2i}{d-1}$. Suppose $q\le n/2$. From $\psi_{n,q,r,t}\neq 0$ and (ii) we get $r\le t$. From $q\le n/2$ and $i=(d-1)(n-r-q)+t(k-2)$ we get
$$ \frac{i}{1-d}+n/2+\frac{t(k-2)}{d-1}\le \frac{i}{1-d}+n-q+\frac{t(k-2)}{d-1}=r. $$
Using $r\le t$ we get 
$$ \frac{i}{1-d}+n/2+\frac{t(k-2)}{d-1}\le t $$
and simplifying yields
$$ n/2\le \frac{i}{d-1}+\frac{t(d+1-k)}{d-1}. $$
Using $k\ge d+1$ we get
$$ n\le \frac{2i}{d-1}. $$ 
(v) Let $k\ge d+2$. Suppose $\psi_{n,q,r,t}\neq 0$ and $q\le tk$. We have to show that $n\le \frac{ki}{k-d-1}$. From $q\le tk$, $i=(d-1)(n-r-q)+t(k-2)$ and $r\le t$ by (iii) we get
$$ \frac{i}{1-d}+n-tk+\frac{t(k-2)}{d-1}\le \frac{i}{1-d}+n-q+\frac{t(k-2)}{d-1}=r\le t. $$
It follows
$$ \frac{i}{1-d}+n-tk+\frac{t(k-2)}{d-1}\le t $$
and then 
$$ n\le \frac{i}{d-1}+t(k+\frac{k-2}{1-d}+1). $$
From (iii) we know $t\le n/k$. It follows
$$ n\le \frac{i}{d-1}+\frac{n}{k}(k+\frac{k-2}{1-d}+1) $$
and then
$$ n(\frac{k-2}{d-1}-1)\le \frac{ki}{d-1}. $$
Using $k\ge d+2$ we get
$$ n\le \frac{\frac{ki}{d-1}}{\frac{k-2}{d-1}-1}=\frac{ki}{k-d-1}. $$
\end{proof}

\begin{Theorem}
Let $d\ge 2$, $k\ge d+1$ and $i\ge 0$. The sequence 
$\{\tilde H^i(\mathcal{M}_{(k,1^{n-k})}^d)\}_n$ stabilizes at $\frac{2i}{d-1}$. If $d$ is even and $k\ge d+2$, the sequence stabilizes at $\frac{ki}{k-d-1}$.
\end{Theorem}
\begin{proof}
From \cite[Theorem 4.4(iii)]{SW}, we have that the characteristic of the
$S_n$-representation on $\tilde H^i(\mathcal{M}_{(k,1^{n-k})}^d)$ is

$$ \sum_{r,t\ge 1,q\ge 0:i=(d-1)(n-r-q)+t(k-2)}\psi_{n,q,r,t} $$
where $\psi_{n,q,r,t}$ is as in the previous lemma.
If $q> n/2$ then we get
$$ \psi_{n,q,r,t}=\psi_{n-1,q-1,r,t}+\square $$
from \ref{psi}(i). From \ref{psi} (iv) we get $\psi_{n,q,r,t}=0$ if $q\le n/2$ and $n>\frac{2i}{d-1}$. Putting these facts together we get for $n>\frac{2i}{d-1}$:
$$ \sum_{r,t\ge 1,q\ge 0:i=(d-1)(n-r-q)+t(k-2)}\psi_{n,q,r,t}=
\sum_{r,t\ge 1,q\ge 1:i=(d-1)(n-r-q)+t(k-2)}\psi_{n,q,r,t}= $$  $$
\sum_{r,t\ge 1,q\ge 1:i=(d-1)(n-r-q)+t(k-2)}\psi_{n-1,q-1,r,t}+\square= 
\sum_{r,t\ge 1,q\ge 0:i=(d-1)(n-1-r-q)+t(k-2)}\psi_{n-1,q,r,t}+\square. $$
Now let $d$ be even and $k\ge d+2$.
If $d$ is even and $q>tk$ we have 
$$ \psi_{n,q,r,t}=\psi_{n-1,q-1,r,t}+\square $$
from \ref{psi}(ii) and $\psi_{n,q,r,t}=0$ if $q\le tk$ and $n>\frac{ki}{k-d-1}$ from \ref{psi} (v). For $n>\frac{ki}{k-d-1}$
the same computation as above yields the stability property.
\end{proof}

\begin{Question}
Is there an explicit formula for the sharp stability bound of $\tilde H^i(\mathcal{M}_{(k,1^{n-k})}^d)$ for general $k,d,i$? 
\end{Question}
~\\ \\
\begin{Table}[Sharp stability bounds for $\tilde H^i(\mathcal{M}_{(k,1^{n-k})}^2)$]
~\\
If $k$ is fixed and $i$ grows, the sequence of bounds appears to increase by $1$ in most of the steps especially at the beginning and with large $k$. Later, there also appear steps with bound differences $2$ or $3$. \\ \\
$k=3:$\\ \\
\begin{tabular}{|*{92}{l|}}\hline
i & $3$ & $4$ & $5$ & $6$ & $7$ & $8$ & $9$ & $10$ & $11$ & $12$ & $13$ & $14$\\ \hline
bound & $6$ & $7$ & $8$ & $11$ & $13$ & $14$ & $16$ & $18$ & $20$ & $21$ & $23$ & $25$\\ \hline
\end{tabular}
~\\ \\ \\
$k=4:$\\ \\
\begin{tabular}{|*{92}{l|}}\hline
i & $5$ & $6$ & $7$ & $8$ & $9$ & $10$ & $11$ & $12$ & $13$ & $14$ & $15$ & $16$\\ \hline
bound & $8$ & $9$ & $10$ & $11$ & $12$ & $15$ & $17$ & $18$ & $19$ & $20$ & $22$ & $24$\\ \hline
\end{tabular}
~\\ \\ \\
$k=5:$\\ \\
\begin{tabular}{|*{92}{l|}}\hline
i & $7$ & $8$ & $9$ & $10$ & $11$ & $12$ & $13$ & $14$ & $15$ & $16$ & $17$ & $18$ & $19$ & $20$\\ \hline
bound & $10$ & $11$ & $12$ & $13$ & $14$ & $15$ & $16$ & $19$ & $21$ & $22$ & $23$ & $24$ & $25$ & $26$\\ \hline
\end{tabular}
~\\ \\ \\
$k=6:$\\ \\
\begin{tabular}{|*{92}{l|}}\hline
i & $9$ & $10$ & $11$ & $12$ & $13$ & $14$ & $15$ & $16$ & $17$ & $18$ & $19$ & $20$ & $21$ & $22$ & $23$\\ \hline
bound & $12$ & $13$ & $14$ & $15$ & $16$ & $17$ & $18$ & $19$ & $20$ & $23$ & $25$ & $26$ & $27$ & $28$ & $29$\\ \hline
\end{tabular}
~\\ \\ \\
$k=7:$\\ \\
\begin{tabular}{|*{92}{l|}}\hline
i & $11$ & $12$ & $13$ & $14$ & $15$ & $16$ & $17$ & $18$ & $19$ & $20$ & $21$ & $22$ & $23$ & $24$ & $25$ & $26$ & $27$	\\ \hline
bound & $14$ & $15$ & $16$ & $17$ & $18$ & $19$ & $20$ & $21$ & $22$ & $23$ & $24$ & $27$ & $29$ & $30$ & $31$ & $32$ & $33$\\ \hline
\end{tabular}
~\\ \\ \\
$k=8:$\\ \\
\begin{tabular}{|*{92}{l|}}\hline
i & $13$ & $14$ & $15$ & $16$ & $17$ & $18$ & $19$ & $20$ & $21$ & $22$ & $23$ & $24$ & $25$ & $26$ & $27$ & $28$ & $29$	\\ \hline
bound & $16$ & $17$ & $18$ & $19$ & $20$ & $21$ & $22$ & $23$ & $24$ & $25$ & $26$ & $27$ & $28$ & $31$ & $33$ & $34$ & $35$\\ \hline
\end{tabular}
~\\ \\ \\
$k=9:$\\ \\
\begin{tabular}{|*{92}{l|}}\hline
i & $15$ & $16$ & $17$ & $18$ & $19$ & $20$ & $21$ & $22$ & $23$ & $24$ & $25$ & $26$ & $27$ & $28$ & $29$ & $30$ & $31$	\\ \hline
bound & $18$ & $19$ & $20$ & $21$ & $22$ & $23$ & $24$ & $25$ & $26$ & $27$ & $28$ & $29$ & $30$ & $31$ & $32$ & $35$ & $37$\\ \hline
\end{tabular}
\end{Table}


\begin{thebibliography}{99}
\bibitem{CF}
T. Church, B. Farb, Representation theory and homological stability, {\it Advances in Mathematics} (2013) 250--314
     \bibitem{C}
T. Church, Homological stability for configuration spaces of manifolds, {\it Invent. Math.} {\bf 188} (2012), no. 2, 465--504

\bibitem{G}
N. Gadish,  Representation stability for families of linear subspace arrangements, {\it Advances in Mathematics} {\bf 322} (2017), 341--377

   \bibitem{HR}
P. Hersh, V. Reiner, Representation stability for cohomology of configuration spaces in $\RR^d$, {\it Int. Math. Res. Not. IMRN} 2017, no. 5, 1433--1486



\bibitem{M}
I.G. Macdonald, Symmetric Functions and Hall Polynomials, 2. ed.,
 Clarendon Press, Oxford, 1995

	

   \bibitem{SW}
    S. Sundaram and V. Welker,
    Group actions on arrangements of linear subspaces 
    and applications to configuration spaces,
    {\it Trans. Amer. Math. Soc.} {\bf 349} (1997), no. 4, 1389--1420.

\end{thebibliography}
\end{document}